\documentclass[a4paper,11pt,reqno]{smfart}
\usepackage{amssymb,amsmath,mathrsfs,graphics,graphicx,mathptm,float}
\usepackage[T1]{fontenc}
\usepackage[english, french]{babel}
\usepackage[all]{xy}
\theoremstyle{plain}
\newtheorem{thm}{Theorem}[section]
\newtheorem{pro}[thm]{Proposition}
\newtheorem{lem}[thm]{Lemma}
\newtheorem{cor}[thm]{Corollary}

\newtheorem{theoalph}{Theorem}

\theoremstyle{definition}

\newtheorem{rem}[thm]{Remark}

\usepackage[colorlinks, linktocpage, 
citecolor = blue, linkcolor = blue]{hyperref}

\addtocounter{section}{0}             
\numberwithin{equation}{section}

\begin{document}
\selectlanguage{english}

\title{Representations of some lattices into the group of analytic diffeomorphisms of the sphere $\mathbb{S}^2$}
\thanks{\noindent The author is supported by the Swiss National Science Foundation grant no PP00P2\_128422 /1 and 
the A.N.R. project ``BirPol``.}

\author{Julie D\'eserti}
\address{Universit\"{a}t Basel, Mathematisches Institut, Rheinsprung $21$, CH-$4051$ Basel, Switzerland.}
\medskip
\address{On leave from Institut de Math\'ematiques de Jussieu, Universit\'e Paris $7$, Projet G\'eom\'etrie et 
Dynamique, Site Chevaleret, Case $7012$, $75205$ Paris Cedex 13, France.}
\email{deserti@math.jussieu.fr}

\maketitle\begin{center}{\today}\end{center}

\begin{abstract}
In \cite{Ghys} it is proved that any morphism from a subgroup of finite index of $\mathrm{SL}(n,\mathbb{Z})$
to the group of analytic diffeomorphisms of $\mathbb{S}^2$ has a finite image as soon as $n\geq 5$. 
The case $n=4$ is also claimed to follow along the same arguments; in fact this is not straightforward and
this case indeed needs a modification of the argument. In this paper we recall the strategy for $n\geq 5$
and then focus on the case $n=4$.

\noindent{\it 2010 Mathematics Subject Classification. --- 58D05, 58B25}
\end{abstract}

\section{Introduction}

After the works of Margulis (\cite{Margulis, VGS}) on the linear representations of lattices of simple, real Lie groups 
with~$\mathbb{R}$-rank larger than $1$, some authors, like Zimmer, suggest to study the actions of lattices on compact manifolds
(\cite{Zimmer3, Zimmer4, Zimmer1, Zimmer2}). One of the main conjectures of this program is the following: let us consider 
a connected, simple, real Lie group $\mathrm{G}$ and let $\Gamma$ be a lattice of $\mathrm{G}$ of $\mathbb{R}$-rank larger than $1$.
If there exists a morphism of infinite image from $\Gamma$ to the group of diffeomorphisms of a compact manifold $M$, 
then the $\mathbb{R}$-rank of $\mathrm{G}$ is bounded by the dimension of $M$. There are a lot of contributions in that
direction (\cite{BurgerMonod, Cantat1, Cantat2, Deserti, FarbShalen, FranksHandel, Ghys, Ghys2, Navas, Polterovich}). 
In this article we will focus on the embeddings of subgroups of finite index of 
$\mathrm{SL}(n,\mathbb{Z})$ into the group $\mathrm{Diff}^\omega(\mathbb{S}^2)$ of real analytic diffeomorphisms of $\mathbb{S}^2$
(\emph{see} \cite{Ghys}).

The article is organized as follows.
First of all we will recall the strategy of \cite{Ghys}: the study of the nilpotent subgroups of 
$\mathrm{Diff}^\omega(\mathbb{S}^2)$ implies that such subgroups
are metabelian. But subgroups of finite index of $\mathrm{SL}(n,\mathbb{Z})$, for $n\geq 5$, contain nilpotent
subgroups of length $n-1$ of finite index which
are not metabelian; as a consequence Ghys gets the following statement.

\begin{theoalph}[\cite{Ghys}]
Let $\Gamma$ be a subgroup of finite index of $\mathrm{SL}(n,\mathbb{Z})$. As soon as $n\geq 5$ there is no
embedding of~$\Gamma$ into $\mathrm{Diff}^\omega(\mathbb{S}^2)$.
\end{theoalph}
 
To study nilpotent subgroups of $\mathrm{Diff}^\omega(\mathbb{S}^2)$ one has to study nilpotent subgroups
of $\mathrm{Diff}^\omega_+(\mathbb{S}^1)$ (\emph{see} \S\ref{Sec:S1}) and then nilpotent subgroups of the 
group of formal diffeomorphisms of $\mathbb{C}^2$ (\emph{see} \S \ref{Sec:C2}).
The last section is devoted to establish the following result.

\begin{theoalph}
Let $\Gamma$ be a subgroup of finite index of $\mathrm{SL}(n,\mathbb{Z})$. As soon as $n\geq 4$ there is no
embedding of~$\Gamma$ into $\mathrm{Diff}^\omega(\mathbb{S}^2)$.
\end{theoalph}

The proof relies on the characterization, up to isomorphism, of nilpotent subalgebras of length $3$ of the 
algebra of formal vector fields of $\mathbb{C}^2$ which vanish at the origin.

\subsection*{Acknowledgements} The author would like to thank Dominique Cerveau and \'Etienne Ghys for interesting discussions 
and advices. 

\section{Nilpotent subgroups of the group of analytic diffeomorphisms of $\mathbb{S}^1$}\label{Sec:S1}

Let $\mathrm{G}$ be a group; let us set 
\begin{align*}
&\mathrm{G}^{(0)}=\mathrm{G} && \& && \mathrm{G}^{(i)}=[\mathrm{G},\mathrm{G}^{(i-1)}] \,\,\,\,\,\forall\, i\geq 1.
\end{align*}
The group $\mathrm{G}$ is {\sl nilpotent} if there exists an integer $n$ such that $\mathrm{G}^{(n)}=\{\mathrm{id}\}$; the 
{\sl length of nilpotence} of $\mathrm{G}$ is the smallest integer $k$ such that $\mathrm{G}^{(k)}=\{\mathrm{id}\}$.

Set
\begin{align*}
&\mathrm{G}_{(0)}=\mathrm{G} && \& && \mathrm{G}_{(i)}=[\mathrm{G}_{(i-1)},\mathrm{G}_{(i-1)}] \,\,\,\,\,\forall\, i\geq 1.
\end{align*}
The group $\mathrm{G}$ is {\sl solvable} if $\mathrm{G}_{(n)}=\{\mathrm{id}\}$ for a certain $n$; the {\sl length
of solvability} of $\mathrm{G}$ is the smallest integer $k$ such that $\mathrm{G}_{(k)}=\{\mathrm{id}\}$.

We say that the group $\mathrm{G}$ (resp. algebra $\mathfrak{g}$) is {\sl metabelian} if $[\mathrm{G},
\mathrm{G}]$ (resp. $[\mathfrak{g},\mathfrak{g}]$) is abelian. 

\begin{pro}[\cite{Ghys}]\label{Pro:nilpsubs1}
Any nilpotent subgroup of $\mathrm{Diff}^\omega_+(\mathbb{S}^1)$ is abelian.
\end{pro}

\begin{proof}
Let $\mathrm{G}$ be a nilpotent subgroup of $\mathrm{Diff}^\omega_+(\mathbb{S}^1)$. Assume that $\mathrm{G}$
is not abelian; it thus contains a Heisenberg group $$\langle f,\,g,\,h\,\vert\, [f,g]=h,\, [f,h]=[g,h]=
\mathrm{id}\rangle.$$
The application ``rotation number`` 
\begin{align*}
& \mathrm{Diff}^\omega_+(\mathbb{S}^1)\to\mathbb{R}/\mathbb{Z}, && \psi\mapsto\lim_{n\to +\infty}\frac{\psi^n(x)-x}{n}
\end{align*}
is not a morphism but its restriction to a solvable subgroup is (\cite{Bavard}). Thus the rotation number of
$h$ is zero and the set $\mathrm{Fix}(h)$ of fixed points of $h$ is non-empty and finite. Considering some
iterates of $f$ and $g$ instead of $f$ and $g$ one can assume that $f$ and $g$ fix any point of $\mathrm{Fix}(h)$. 
The set of fixed points of a non trivial element of $\langle f,\,g \rangle$ is finite and invariant by $h$
so the action of $\langle f,\,g\rangle$ is free\footnote{The stabilizer of every point is trivial, {\it i.e.}
the action of a non trivial element of $\langle f,\,g\rangle$ has no fixed point.} on each component of 
$\mathbb{S}^1\setminus\mathrm{Fix}(h)$. But the action of a free group on $\mathbb{R}$ is abelian: 
contradiction.
\end{proof}

\section{Nilpotent subgroups of the group of formal diffeomorphisms of $\mathbb{C}^2$}\label{Sec:C2}

Let us denote $\widehat{\mathrm{Diff}}(\mathbb{C}^2,0)$ the group of formal diffeomorphisms of $\mathbb{C}^2$, 
{\it i.e.} the formal completion of the group of germs of holomorphic diffeomorphisms at $0$.
For any $i$ let $\mathrm{Diff}_i$ be the quotient of $\widehat{\mathrm{Diff}}(\mathbb{C}^2,0)$ by the 
normal subgroups of formal diffeomorphisms tangent to the identity with multiplicity $i$; it can be
viewed as the set of jets of diffeomorphisms at order $i$ with the law of composition with truncation 
at order $i$. Note that~$\mathrm{Diff}_i$ is a complex linear algebraic group. One can see 
$\widehat{\mathrm{Diff}}(\mathbb{C}^2,0)$ as the projective limit of the $\mathrm{Diff}_i$'s: $\widehat{\mathrm{Diff}}(\mathbb{C}^2,0)=\displaystyle
\lim_{\leftarrow}\mathrm{Diff}_i$. 
Let us denote by $\widehat{\chi}(\mathbb{C}^2,0)$ the algebra of formal vector fields in $\mathbb{C}^2$
vanishing at $0$. One can define the set $\chi_i$ of the $i$-th jets of vector fields; one has
$\displaystyle\lim_{\leftarrow}\chi_i=\widehat{\chi}(\mathbb{C}^2,0)$.

Let $\widehat{\mathcal{O}}(\mathbb{C}^2)$ be the ring of formal series in two variables and let 
$\widehat{K}(\mathbb{C}^2)$ be its fraction field; 
$\mathcal{O}_i$ is the set of elements of $\widehat{\mathcal{O}}
(\mathbb{C}^2)$ truncated at order $i$.

The family $\big(\exp_i\colon\chi_i\to\mathrm{Diff}_i\big)_i$ is filtered, {\it i.e.} compatible with 
the truncation. We then define the exponential application as follows: $\exp=\displaystyle
\lim_{\leftarrow}\exp_i\colon\widehat{\chi}(\mathbb{C}^2,0)\to\widehat{\mathrm{Diff}}(\mathbb{C}^2,0)$.

As in the classical case, if $X$ belongs to $\widehat{\chi}(
\mathbb{C}^2,0)$, then $\exp(X)$ can be seen as the ``flow at time $1$'' of $X$. Indeed 
an element $X_i$ of $\chi_i$ can be seen as a derivation of $\mathcal{O}_i$; so it can 
be written $S_i+N_i$ where $S_i$ and $N_i$ are two semi-simple, resp. nilpotent derivations 
which commute. Taking the limit, one gets $X=S+N$ where $S$ is a semi-simple
vector field and $N$ a nilpotent one and $[S,N]=\mathrm{id}$ (\emph{see} \cite{Martinet}).
A semi-simple vector field is a formal vector field conjugate to a diagonal linear 
vector field which is complete. A vector field is nilpotent if and only if its linear part 
is; let us remark that the usual flow $\varphi_t$ of a nilpotent vector field is 
polynomial in~$t$
\begin{align*}
&\varphi_t(x)=\sum_IP_I(t)x^I,&&P_I\in(\mathbb{C}[t])^2
\end{align*}
so $\varphi_1(x)$ is well defined. As a consequence $\exp(tX)=\exp(tS)\exp(tN)$ is 
well defined for $t=1$. Note that the Jordan decomposition is purely formal: if $X$
is holomorphic, $S$ and $N$ are not necessary holomorphic.

\begin{pro}[\cite{Ghys}]\label{Pro:nilsubalg}
Any nilpotent subalgebra of $\widehat{\chi}(\mathbb{C}^2,0)$ is metabelian.
\end{pro}

\begin{proof}
Let $\mathfrak{l}$ be a nilpotent subalgebra of $\widehat{\chi}(\mathbb{C}^2,0)$ and let
$Z(\mathfrak{l})$ be its center. Since $\widehat{\chi}(\mathbb{C}^2,0)\otimes
\widehat{K}(\mathbb{C}^2)$ is a vector space of dimension $2$ over $\widehat{K}(\mathbb{C}^2)$
one has the following alternative:
\smallskip
\begin{itemize}
\item[$\bullet$] the dimension of the subspace generated by $Z(\mathfrak{l})$ in 
$\widehat{\chi}(\mathbb{C}^2,0)\otimes\widehat{K}(\mathbb{C}^2)$ is $1$;
\smallskip
\item[$\bullet$] the dimension of the subspace generated by $Z(\mathfrak{l})$ in 
$\widehat{\chi}(\mathbb{C}^2,0)\otimes\widehat{K}(\mathbb{C}^2)$ is $2$.
\end{itemize}
\smallskip
Let us study these different cases.

Under the first assumption there exists an element $X$ of $Z(\mathfrak{l})$ having 
the following property: any vector field of $Z(\mathfrak{l})$ can be written $uX$ 
with $u$ in $\widehat{K}(\mathbb{C}^2)$. Let us consider the subalgebra $\mathfrak{g}$ of 
$\mathfrak{l}$ given by $$\mathfrak{g}=\big\{\widetilde{X}\in\mathfrak{l}\,\vert\,\exists\,u\in\widehat{K}(\mathbb{C}^2)
,\, \widetilde{X}=uX\big\}.$$ Since $X$ belongs to $Z(\mathfrak{l})$, the algebra $\mathfrak{g}$
is abelian; it is also an ideal of $\mathfrak{l}$. Let us assume that $\mathfrak{l}$ is 
not abelian: let~$Y$ be an element of $\mathfrak{l}$ whose projection on $\mathfrak{l}/
\mathfrak{g}$ is non trivial and central. Any vector field of $\mathfrak{l}$ can be written
as $uX+vY$ with $u$, $v$ in $\widehat{K}(\mathbb{C}^2)$. As $X$ belongs to $Z(\mathfrak{l})$
and $Y$ is central modulo $\mathfrak{g}$ one has $$X(u)=X(v)=Y(v)=0.$$ The vector fields 
$\frac{\partial}{\partial x}$ and $\frac{\partial}{\partial y}$ being some linear combinations
of $X$ and $Y$ with coefficients in $\widehat{K}(\mathbb{C}^2,0)$, the partial derivatives of 
$v$ are zero so $v$ is a constant. Therefore $[\mathfrak{l},\mathfrak{l}]\subset\mathfrak{g}$; 
but $\mathfrak{g}$ is abelian so $\mathfrak{l}$ is metabelian.

In the second case $Z(\mathfrak{l})$ has two elements $X$ and $Y$ which are 
linearly independent on $\widehat{K}(\mathbb{C}^2)$. Any vector field of $\mathfrak{l}$
can be written as $uX+vY$ with $u$ and $v$ in $\widehat{K}(\mathbb{C}^2)$. Since $X$ and 
$Y$ belong to $Z(\mathfrak{l})$ one has $$X(u)=X(v)=Y(u)=Y(v)=0.$$ As a consequence
$u$ and $v$ are constant, {\it i.e.} $\mathfrak{l}\subset\{uX+vY\,\vert\, u,\, v\in
\mathbb{C}\}$; in particular $\mathfrak{l}$ is abelian.
\end{proof}

\begin{pro}[\cite{Ghys}]\label{Pro:nilpmet}
Any nilpotent subgroup of $\widehat{\mathrm{Diff}}(\mathbb{C}^2,0)$ is metabelian.
\end{pro}

\begin{proof}
Let $\mathrm{G}$ be a nilpotent subgroup of $\widehat{\mathrm{Diff}}(\mathbb{C}^2,0)$ of 
length $k$. Let us denote $\mathrm{G}_i$ the projection of $\mathrm{G}$ on~$\mathrm{Diff}_i$.
The Zariski closure $\overline{\mathrm{G}_i}$ of $\mathrm{G}_i$ in $\mathrm{Diff}_i$ is 
an algebraic nilpotent subgroup of length $k$. It is sufficient to prove that $\overline{\mathrm{G}_i}$ is metabelian.

Since $\overline{\mathrm{G}_i}$ is a complex algebraic subgroup it is the direct product 
of the subgroup $\overline{\mathrm{G}_{i,u}}$ of its unipotent elements and the subgroup 
$\overline{\mathrm{G}_{i,s}}$ of its semi-simple elements (\emph{see for example} \cite{Borel}).

An element of $\mathrm{Diff}_i$ is unipotent if and only if its linear part, which 
is in $\mathrm{GL}(2,\mathbb{C})$, is; so $\overline{\mathrm{G}_{i,s}}$ projects 
injectively onto a nilpotent subgroup of $\mathrm{GL}(2,\mathbb{C})$. 
Therefore $\overline{\mathrm{G}_{i,s}}$ is abelian.

Let us now consider $\overline{\mathrm{G}_{i,u}}$; this group is the exponential of 
a nilpotent Lie algebra $\mathfrak{l}_i$ of $\chi_i$ of length $k$. Taking the limit
one thus obtains the existence of a nilpotent subalgebra $\mathfrak{l}$ of 
$\widehat{\chi}(\mathbb{C}^2,0)$ of length $k$ such that $\exp(\mathfrak{l})$
projects onto $\overline{\mathrm{G}_{i,u}}$ for any $i$. According to 
Proposition \ref{Pro:nilsubalg} the subalgebra $\mathfrak{l}$ and thus $\overline{\mathrm{G}_{i,u}}$
are metabelian.  
\end{proof}

\section{Nilpotent subgroups of the group of analytic diffeomorphisms of $\mathbb{S}^2$}\label{Sec:S2}

\begin{pro}[\cite{Ghys}]\label{Pro:nilpfiniteorbit}
 Any nilpotent subgroup of $\mathrm{Diff}^\omega(\mathbb{S}^2)$ has a finite orbit.
\end{pro}

\begin{proof}
Let $\mathrm{G}$  be a nilpotent subgroup of $\mathrm{Diff}^\omega(\mathbb{S}^2)$; up to finite 
index one can assume that the elements of~$\mathrm{G}$ preserve the orientation. Let $\phi$ be a 
non trivial element of $\mathrm{G}$ which commutes with $\mathrm{G}$. Let $\mathrm{Fix}(\phi)$
be the set of fixed points of $\phi$; it is a non empty analytic subspace of 
$\mathbb{S}^2$ invariant by $\mathrm{G}$. If $p$ is an isolated fixed point of $\phi$, then the 
orbit of $p$ under the action of $\mathrm{G}$ is finite. So it is sufficient to study the case
where $\mathrm{Fix}(\phi)$ only contains curves; there are thus two possibilities:
\begin{itemize}
\item[$\bullet$] $\mathrm{Fix}(\phi)$ is a singular analytic curve whose set of singular points 
is a finite orbit for $\mathrm{G}$;

\item[$\bullet$] $\mathrm{Fix}(\phi)$ is a smooth analytic curve, not necessary connected. One 
of the connected component of~$\mathbb{S}^2\setminus~\mathrm{Fix}(\phi)$ is a disk denoted 
$\mathbb{D}$. Any subgroup $\Gamma$ of finite index of $\mathrm{G}$ which contains $\phi$
fixes $\mathbb{D}$. Let us consider an element $\gamma$ of $\Gamma$ and a fixed point $m$
of $\gamma$ which is in $\overline{\mathbb{D}}$. By construction $\phi$ has no fixed point 
in $\mathbb{D}$ so according to the Brouwer Theorem $(\phi^k(m))_k$ has a limit
point on the boundary $\partial\mathbb{D}$ of $\overline{\mathbb{D}}$. Therefore $\gamma$
has at least one fixed point on $\partial\mathbb{D}$. The group $\Gamma$ thus acts on 
the circle $\partial\mathbb{D}$ and any of its elements has a fixed point on 
$\mathbb{D}$. Then $\Gamma$ has a fixed point on $\partial\mathbb{D}$ (Proposition
\ref{Pro:nilpsubs1}). 
\end{itemize}\end{proof}

\begin{thm}[\cite{Ghys}]
Any nilpotent subgroup of $\mathrm{Diff}^\omega(\mathbb{S}^2)$ is metabelian. 
\end{thm}

\begin{proof}
Let $\mathrm{G}$ be a nilpotent subgroup of $\mathrm{Diff}^\omega(\mathbb{S}^2)$ and let 
$\Gamma$ be a subgroup of finite index of $\mathrm{G}$ having a fixed point $m$ (such a 
subgroup exists according to Proposition \ref{Pro:nilpfiniteorbit}). One can embed
$\Gamma$ into $\widehat{\mathrm{Diff}}(\mathbb{R}^2,0)$, and so into $\widehat{\mathrm{Diff}}
(\mathbb{C}^2,0)$, by considering the jets of infinite order of elements of $\Gamma$ in 
$m$. According to Proposition~\ref{Pro:nilpmet} the group $\Gamma$ is metabelian.

One can assume that $\mathrm{G}$ is a finitely generated group.

Let us first assume that $\mathrm{G}$ has no element of finite order. Then $\mathrm{G}$
is a cocompact lattice of the nilpotent, simply connected Lie group $\mathrm{G}
\otimes\mathbb{R}$ (\emph{see} \cite{Raghunathan}). The group $\mathrm{G}$ is metabelian if and 
only if $\mathrm{G}\otimes\mathbb{R}$ is; but $\Gamma$ is metabelian so $\mathrm{G}
\otimes\mathbb{R}$ also.

Finally let us consider the case where $\mathrm{G}$ has at least one element of finite order. 
The set of such elements is a normal subgroup of $\mathrm{G}$ which thus intersects
non trivially the center $Z(\mathrm{G})$ of $\mathrm{G}$. Let us consider a non trivial element $\phi$ of 
$Z(\mathrm{G})$ which has finite order. Let us recall that a finite group of diffeomorphisms of
the sphere is conjugate to a group of isometries. Denote by $\mathrm{G}^+$
the subgroup of elements of $\mathrm{G}$ which preserve the orientation. It is thus
sufficient to prove that $\mathrm{G}^+$ is metabelian; indeed if $\phi$ does not preserve
the orientation, $\phi$ has order $2$ and $\mathrm{G}=\mathbb{Z}/
2\mathbb{Z}\times\mathrm{G}^+$. So let us assume that $\phi$ preserves the 
orientation; $\phi$ is conjugate to a direct isometry of $\mathbb{S}^2$ and has
exactly two fixed points on the sphere. The group $\mathrm{G}$ has thus an invariant set of two 
elements. By considering germs in the neighborhood of these two
points, one gets that $\mathrm{G}$ can be embedded into $2\cdot \mathrm{Diff}
(\mathbb{R}^2,0)$\footnote{Let $\mathrm{G}$ be a group and let $q$ be a 
positive integer; $q\cdot \mathrm{G}$ denotes the semi-direct product of 
$\mathrm{G}^q$ by $\mathbb{Z}/q\mathbb{Z}$ under the action of the cyclic
permutation of the factors.} and thus into $2\cdot\mathrm{Diff}(\mathbb{C}^2,0)$:
$$1\longrightarrow\mathrm{Diff}(\mathbb{C}^2,0)\longrightarrow 2\cdot
\mathrm{Diff}(\mathbb{C}^2,0)\longrightarrow\mathbb{Z}/2\mathbb{Z}\longrightarrow
0.$$
Let us remark that $2\cdot\mathrm{Diff}(\mathbb{C}^2,0)$ is the projective
limit of the algebraic groups $2\cdot\mathrm{Diff}_i$. The end of the proof
is thus the same as the proof of Proposition \ref{Pro:nilpmet} except that the subgroup 
of the semi-simple elements of $2\cdot \mathrm{Diff}_i$ embeds now in $2\cdot
\mathrm{GL}(2,\mathbb{C})$; it is metabelian because it contains an abelian 
subgroup of index~$2$.
\end{proof}

Let $\Gamma$ be a subgroup of finite index of $\mathrm{SL}(n,\mathbb{Z})$ for $n\geq 5$. 
Since $\Gamma$ contains nilpotent subgroups of finite index of length $n-1$ (for 
example the group of upper triangular unipotent matrices) which are not metabelian
one gets the following statement.

\begin{cor}[\cite{Ghys}]
Let $\Gamma$ be a subgroup of finite index of $\mathrm{SL}(n,\mathbb{Z})$; as soon as $n\geq 5$ 
there is no embedding of $\Gamma$ into $\mathrm{Diff}^\omega(\mathbb{S}^2)$. 
\end{cor}

\section{Nilpotent subgroups of length $3$ of the group of analytic diffeomorphisms of $\mathbb{S}^2$}
 
Let us precise Proposition \ref{Pro:nilsubalg} for nilpotent subalgebras of length $3$ of $\widehat{\chi}
(\mathbb{C}^2,0)$. Let $\mathfrak{l}$ be such an algebra. The center $Z(\mathfrak{l})$ of $\mathfrak{l}$
generates a subspace of dimension at most $1$ of $\widehat{\chi}(\mathbb{C}^2,0)\otimes\widehat{K}
(\mathbb{C}^2)$, for else $\mathfrak{l}$ would be abelian (Proposition \ref{Pro:nilsubalg}) and this is impossible under our 
assumptions. So let us assume that the dimension of the subspace generated by $Z(\mathfrak{l})$ in 
$\widehat{\chi}(\mathbb{C}^2,0)\otimes\widehat{K}(\mathbb{C}^2)$ is $1$. There exists an element $X$ in 
$Z(\mathfrak{l})$ with the following property: any element of~$Z(\mathfrak{l})$ can be written $uX$ with $u$ in 
$\widehat{K}(\mathbb{C}^2)$. Let $\mathfrak{g}$ denote the abelian ideal of $\mathfrak{l}$ 
defined by $$\mathfrak{g}=\big\{\widetilde{X}\in\mathfrak{l}\,\big\vert\, \exists\, u\in\widehat{K}(\mathbb{C}^2),
\, \widetilde{X}=uX\big\}.$$ By hypothesis $\mathfrak{l}$ is not abelian. Let $Y$ be in $\mathfrak{l}$; assume that
its projection onto $\mathfrak{l}/\mathfrak{g}$ is a non trivial element of~$Z(\mathfrak{l}/\mathfrak{g})$.
Any vector field of $\mathfrak{l}$ can be written 
\begin{align*}
& uX+vY, && u,\,v\in \widehat{K}(\mathbb{C}^2).
\end{align*}
Since $X$, resp. $Y$ belongs to $Z(\mathfrak{l})$, resp. $Z(\mathfrak{l}/\mathfrak{g})$ and since the
length of $\mathfrak{l}$ is $3$, one has 
\begin{equation}\label{eqlength3}
X(u)=Y^3(u)=X(v)=Y(v)=0.
\end{equation}

 If $X$ and $Y$ are non singular, one can choose formal coordinates $x$ and $y$ such that 
$X=\frac{\partial}{\partial x}$ and~$Y=\frac{\partial}{\partial y}$. The previous conditions
can be thus translated as follows: $v$ is a constant and $u$ is a polynomial in $y$ of degree $2$.
We will see that we have a similar property without assumption on $X$ and $Y$.

\begin{lem}
Let $X$ and $Y$ be two vector fields of $\widehat{\chi}(\mathbb{C}^2,0)$ that commute and are not colinear.
One can assume that $(X,Y)=\Big(\frac{\partial}{\partial\tilde{x}},\frac{\partial}{\partial
\tilde{y}}\Big)$ where $\widetilde{x}$ and $\widetilde{y}$ are two independent variables in 
a Liouvillian extension of~$\widehat{K}(\mathbb{C}^2,0)$.
\end{lem}

\begin{proof}
Since $X$ and $Y$ are non colinear, there exist two $1$-forms $\alpha$, $\beta$ with coefficients
in $\widehat{K}(\mathbb{C}^2)$ such that 
\begin{align*}
& \alpha(X)=1, && \alpha(Y)=0, && \beta(X)=0, && \beta(X)=1. 
\end{align*}
The vector fields $X$ and $Y$ commute if and only if $\alpha$ and $\beta$ are closed (this 
statement of linear algebra is true for convergent meromorphic vector fields and is also true
in the completion). The $1$-form $\alpha$ is closed so according to \cite{CerveauMattei} one has
$$\alpha=\sum_{i=1}^r\lambda_i\frac{d\widehat{\phi}_i}{\widehat{\phi}_i}+d\Big(\frac{\widehat{\psi}_1}{
\widehat{\psi}_2}\Big)=d\Big(\sum_{i=1}^r\lambda_i\log\widehat{\phi}_i+\frac{\widehat{\psi}_1}{
\widehat{\psi}_2}\Big)$$ where $\widehat{\psi}_1$, $\widehat{\psi}_2$ and the $\widehat{\phi}_i$ denote
some formal series and the $\lambda_i$ some complex numbers. One has a similar expression for 
$\beta$. So there exists a Liouvillian extension $\kappa$ of $\widehat{K}(\mathbb{C}^2)$ having two
elements $\widetilde{x}$ and $\widetilde{y}$ with $\alpha=d\widetilde{x}$ and $\beta=d\widetilde{y}$.
One thus has 
\begin{align*}
&X(\widetilde{x})=1, &&X(\widetilde{y})=0, &&Y(\widetilde{x})=0, &&Y(\widetilde{y})=1. 
\end{align*}
\end{proof}

From (\ref{eqlength3}) one gets: $v$ is a constant and $u$ is a polynomial in 
$\widetilde{y}$ of degree $2$; so one proves the following statement.

\begin{pro}\label{Pro:nilpsubalgl3}
Let $\mathfrak{l}$ be a nilpotent subalgebra of $\widehat{\chi}(\mathbb{C}^2,0)$ of length $3$. 
Then $\mathfrak{l}$ is isomorphic to a subalgebra~of 
$$\mathfrak{n}=\Big\{P(\widetilde{y})\frac{\partial}{\partial\widetilde{x}}+\alpha
\frac{\partial}{\partial\widetilde{y}}\,\,\Big\vert\,\,\alpha\in\mathbb{C},\, P\in\mathbb{C}[\widetilde{y}],\,\deg P=2\Big\}.$$ 
\end{pro}

\begin{rem}
 We use a real version of this statement whose proof is an adaptation of the previous one:
a nilpotent subalgebra $\mathfrak{l}$ of length $3$ of $\widehat{\chi}(\mathbb{R}^2,0)$ 
is isomorphic to a subalgebra of 
$$\mathfrak{n}=\Big\{P(\widetilde{y})\frac{\partial}{\partial\tilde{x}}+\alpha\frac{\partial}{\partial\tilde{y}}\,
\Big\vert\,\alpha\in\mathbb{R},\, P\in\mathbb{R}[\widetilde{y}],\,\deg P=2\Big\}.$$
\end{rem}

\begin{thm}
Let $\Gamma$ be a subgroup of finite index of $\mathrm{SL}(n,\mathbb{Z})$; as soon as $n\geq 4$ 
there is no embedding of~$\Gamma$ into $\mathrm{Diff}^\omega(\mathbb{S}^2)$. 
\end{thm}

\begin{proof}
Let $\mathrm{U}(4,\mathbb{Z})$ (resp. $\mathrm{U}(4,\mathbb{R})$) be the subgroup of unipotent upper triangular
matrices of $\mathrm{SL}(4,\mathbb{Z})$ (resp.~$\mathrm{SL}(4,\mathbb{R})$); it is a nilpotent subgroup of 
length $3$. Assume that there exists an embedding from a subgroup $\Gamma$ of finite index of $\mathrm{SL}(4,
\mathbb{Z})$ into $\mathrm{Diff}^\omega(\mathbb{S}^2)$. Up to finite index $\Gamma$ contains $\mathrm{U}(4,
\mathbb{Z})$. Let us set~$\mathrm{H}=~\rho(\mathrm{U}(4,\mathbb{Z}))$. Up to finite index $\mathrm{H}$ has 
a fixed point (Proposition \ref{Pro:nilpfiniteorbit}). One can thus see $\mathrm{H}$ as a subgroup of
$\mathrm{Diff}(\mathbb{R}^2,0)\subset\widehat{\mathrm{Diff}}(\mathbb{R}^2,0)$ up to finite index.

Let us denote $j^1$ the morphism from $\widehat{\mathrm{Diff}}(\mathbb{R}^2,0)$ to $\mathrm{Diff}_i$.
Up to conjugation $j^1(\rho(\mathrm{U}(4,\mathbb{Z})))$ is a subgroup of $$\Big\{\left[\begin{array}{cc}
\lambda & t \\
0 & \lambda                                                                                        
\end{array}
\right]\,\Big\vert\,\lambda\in\mathbb{R}^*,\, t\in\mathbb{R}\Big\}.$$ Up to index $2$ one can thus
assume that $j^1\circ\rho$ takes values in the connected, simply connected group $\mathrm{T}$ 
defined by $$\mathrm{T}=\Big\{\left[\begin{array}{cc}
\lambda & t \\
0 & \lambda                                                                                        
\end{array}
\right]\,\Big\vert\,\lambda,\,t\in\mathbb{R},\,\lambda>0\Big\}.$$ Let us set 
$$\mathrm{Diff}_i(\mathrm{T})=\big\{f\in\mathrm{Diff}_i\,\vert\, j^1(f)\in\mathrm{T}\big\};$$ 
the group $\mathrm{Diff}_i(\mathrm{T})$ is a connected, simply connected, nilpotent and algebraic group.
The morphism $$\rho_i\colon\mathrm{U}(4,\mathbb{Z})\to\mathrm{Diff}_i$$ can be extended to a unique
continuous morphism $\widetilde{\rho_i}\colon\mathrm{U}(4,\mathbb{R})\to\mathrm{Diff}_i(\mathrm{T})$ 
(\emph{see} \cite{Malcev1, Malcev2}) so to an algebraic morphism\footnote{Let $\mathrm{N}_1$ and $\mathrm{N}_2$ be two connected, 
simply connected, nilpotent and algebraic subgroups on $\mathbb{R}$; any continuous morphism 
between $\mathrm{N}_1$ and $\mathrm{N}_2$ is algebraic.}. Let us note that $\widetilde{\rho_i}
(\mathrm{U}(4,\mathbb{Z}))$ is an algebraic subgroup of $\mathrm{Diff}_i(\mathrm{T})$ which contains
$\rho_i(\mathrm{U}(4,\mathbb{Z}))$; in particular $\overline{\mathrm{H}_i}=\overline{\rho_i(\mathrm{U}(4,
\mathbb{Z}))}\subset\widetilde{\rho_i}(\mathrm{U}(4,\mathbb{R}))$. By construction the family $(\mathrm{H}_i)_i$
is filtered; since the extension is unique, the family $(\widetilde{\rho_i})_i$ is also filtered. Therefore
$\displaystyle\mathrm{K}=\lim_{\leftarrow}\overline{\mathrm{H}_i}$ is well defined. Since $\rho$ is injective, $\mathrm{H}$
is a nilpotent subgroup of length $3$; as $\mathrm{H}\subset\mathrm{K}$ and as any $\overline{\mathrm{H}_i}$ is 
nilpotent of length at most $3$ the group $\mathrm{K}$ is nilpotent of length at most $3$. For $i$ sufficiently 
large $\widetilde{\rho_i}(\mathrm{U}(4,\mathbb{R}))$ is nilpotent of length $3$; this group is connected so 
its Lie algebra is also nilpotent of length $3$. Therefore the image of 
$$D\widetilde{\rho}:=\lim_{\leftarrow}D\widetilde{\rho_i}\colon\mathfrak{u}(4,\mathbb{R})\to\widehat{\chi}
(\mathbb{R}^2,0)$$ is isomorphic to $\mathfrak{n}$ (Proposition \ref{Pro:nilpsubalgl3}). So there exists 
a surjective map $\psi$ from $\mathfrak{u}(4,\mathbb{R})$ onto $\mathfrak{n}$. The kernel of~$\psi$ is 
an ideal of $\mathfrak{u}(4,\mathbb{R})$ of dimension $2$; hence $\ker\psi=\langle \delta_{14},a\delta_{13}
+b\delta_{24}\rangle$ where the $\delta_{ij}$ denote the Kronecker matrices. One concludes by remarking that 
$\dim Z(\mathfrak{u}(4,\mathbb{R})/\ker\psi)=2$ whereas $\dim Z(\mathfrak{n})=1$. 
\end{proof}

\begin{cor}
The image of a morphism from a subgroup of $\mathrm{SL}(n,\mathbb{Z})$ of finite index to $\mathrm{Diff}^\omega
(\mathbb{S}^2)$ is finite as soon as $n\geq 4$.
\end{cor}

\vspace{8mm}

\bibliographystyle{plain}
\bibliography{biblio}
\nocite{}

\end{document}